%% file: paper.tex
\documentclass[a4paper]{amsart}

\usepackage[T1]{fontenc}
\usepackage[dvipsnames,svgnames,table]{xcolor}
\usepackage{mathtools}
\usepackage{comment}
\usepackage[inline]{enumitem}
\usepackage{amssymb}
\usepackage[foot]{amsaddr}

\usepackage[unicode=true]{hyperref}
\hypersetup{
    colorlinks,
    linkcolor={RoyalBlue},
    citecolor={RubineRed},
    urlcolor={blue!80!black},
    pdftitle={Relative dimension}
}

\numberwithin{equation}{section}

\newcommand{\jedrzej}[1]{\textcolor{ProcessBlue}{J\k{e}drzej: #1}}
\newcommand{\dominik}[1]{\textcolor{DarkRed}{Dominik: #1}}

\DeclarePairedDelimiter\rsize{\|}{\|_\mathrm{r}}

\input{macros}

\setenumerate{label=\textup{(\roman*)}, noitemsep, topsep=3pt-\parskip,
labelindent=.2em, leftmargin=*, widest=iii,}
\setitemize{noitemsep, topsep=-\parskip, labelindent=.2em, leftmargin=*, widest=iii,}

\author[D. Dürrschnabel]{Dominik Dürrschnabel}
\address[D. Dürrschnabel]{School of Engineering, Baden-Wuerttemberg Cooperative State University (DHBW) Mosbach, Germany}
\email{dominik.duerrschnabel@mosbach.dhbw.de}

\author[J. Hodor]{J\k{e}drzej Hodor}
\address[J. Hodor]{Theoretical Computer Science Department,
  Faculty of Mathematics and Computer Science and  Doctoral School of Exact and Natural Sciences, Jagiellonian University, Krak\'ow, Poland}
\email{jedrzej.hodor@gmail.com}

\author[P. Micek]{Piotr Micek}
\address[P. Micek]{Theoretical Computer Science Department, Faculty of Mathematics and Computer Science, Jagiellonian University, Kraków,
  Poland}
\email{piotr.micek@uj.edu.pl}

\author[G. Stumme]{Gerd Stumme}
\address[G. Stumme]{Knowledge and Data Engineering Group,  Department of Electrical Engineering and Computer Science,  University of Kassel, Germany}
\email{stumme@cs.uni-kassel.de}

\author[W. T. Trotter]{William T.\ Trotter}
\address[W. T. Trotter]{School of Mathematics, Georgia Institute of Technology, Atlanta, Georgia 30332}
\email{trotter@math.gatech.edu}

\thanks{J.~Hodor and P.~Micek are supported by the National Science Centre, Poland under grant UMO-2022/47/B/ST6/02837 within the OPUS 24 program.}

\begin{document}

\title[Relative Dimension of Posets]{Relative Dimension of Posets}

\date{}

\dedicatory{}

\begin{abstract}
  Dimension of partially ordered sets (posets for short) can be seen as a measure of how much space is needed to store posets.
  Refining the definition of dimension, or actually the local dimension, in this respect we obtain the notion of relative dimension.
  We discuss properties of relative dimension and we give bounds for relative dimension of some well-known families of posets.
\end{abstract}

\maketitle

\section{Introduction}

An $n$-element poset $P$ can be represented by its order-relation matrix, which uses $\Oh(n^2)$ bits. In 1941, Dushnik and Miller~\cite{DusMil41} introduced the dimension of a poset. 
If $\dim(P)\le d$, then $P$ can instead be encoded by $d$ linear orders, using $\Oh(dn\log n)$ bits. This is an improvement whenever $d=o(n/\log n)$. The condition need not hold: Erd\H{o}s, Kierstead, and Trotter~\cite{Erdös.1991} proved that almost all labeled $n$-element posets have dimension
$\frac n4-\Theta(n/{\log n})$.
Several variants of classical dimension have been introduced. Boolean dimension, developed in the late 1980s by Ne\v{s}et\v{r}il and Pudl\'{a}k~\cite{NP89}, retains a comparable encoding property and can be substantially smaller than the classical dimension on natural classes of posets. 
Another variant of dimension, which can again be substantially smaller than the classical dimension and has the encoding property, is the local dimension introduced by Ueckerdt in 2016~\cite{Ueckerdt16}; see also~\cite{Trotter.2017,BGT20,Barrera.2020,DBLP:journals/ejc/KimMMSSUW20}. 
We refine local dimension to relative dimension, which better fits the framework of data compression.
Local dimension counts the maximum number of occurrences of an element in the representation of a poset, whereas relative dimension counts the average number. 

We begin with the necessary definitions, then introduce dimension, local dimension, and relative dimension, together with examples illustrating the new parameter. 
\Cref{sec:results} states our main results, 
\Cref{sec:proofs} contains their proofs, and \Cref{sec:open-problems} concludes with open problems.

\subsection{Basic definitions}

For each positive integer $d$, let $\defin{[d]}=\{1,\dots, d\}$, and set $\defin{[0]}=\emptyset$. 
All logarithms in this paper are base $2$.
We use the convention that $\max \emptyset = 1$.

A \defin{partially ordered set}, or \defin{poset}, is a pair $(P,\le)$, 
where $P$ is a finite set and $\le$ is a reflexive, antisymmetric, and transitive relation on $P$. 
Thus, for all $x,y,z\in P$, we have $x\le x$; 
if $x\le y$ and $y\le x$, then $x=y$; and if $x\le y$ and $y\le z$, then $x\le z$.

Two elements $x,y\in P$ are \defin{comparable} if $x\le y$ or $y\le x$; 
otherwise they are \defin{incomparable}, denoted $x\parallel y$. For $S\subseteq P$, the \defin{subposet} of $P$ induced by $S$ has ground set $S$ and satisfies $x\le y$ in $S$ if and only if $x\le y$ in $P$, for all $x,y\in S$. A poset is a \defin{chain} if every pair of elements is comparable and an \defin{antichain} if no two distinct elements are comparable. A \defin{linear extension} of $P$ is a linear order on $P$ that extends its partial order. We write \defin{$|P|$} for the number of elements of $P$.

\subsection{Dimension}

When $d$ is a positive integer, a set $\mathcal{L} = \{L_1,\dots,L_d\}$ of linear extensions of a poset $P$ is called a \defin{realizer} of $P$ when for all $x$ and $y$ in $P$, 
\[\text{$x\le y$ in $P$ \ \ if and only if \ \ $x\le y$ in $L_i$ for each $i\in[d]$.}\]
The \defin{dimension} of a poset $P$, denoted \defin{$\dim(P)$}, is the least size of a realizer that $P$ admits.

\subsection{Local Dimension}

Let $P$ be a poset.
A \defin{partial linear extension} of $P$ is a linear extension of a subposet of $P$.
Let $t$ be a nonnegative integer and let $\mathcal{L} = \{M_1,\dots,M_t\}$ be a set of partial linear extensions of $P$.
The set $\mathcal{L}$ is a \defin{local realizer} of $P$ if for all $x$ and $y$ in $P$ there exists $i \in [t]$ with $x,y \in M_i$, and
\[\text{$x\le y$ in $P$ \ \ if and only if \ \ $x\le y$ in $M_i$ for each $i\in[t]$ such that $x,y \in M_i$.}\]
The \defin{frequency} of a local realizer $\mathcal{L} = \{M_1,\dots,M_t\}$ of $P$ is the maximum over all $x$ in $P$ over the number of occurrences in $\mathcal{L}$, that is,
\[\mathrm{\defin{\|\calL\|}} = \max_{x \in P}|\{i \in [t] : x \in M_i\}|.\]
The \defin{local dimension} of $P$, denoted \defin{$\ldim(P)$}, is the least frequency of a local realizer that $P$ admits.

\subsection{Relative Dimension}
The \defin{relative frequency} of a local realizer $\mathcal{L} =  \{M_1,\dots,M_t\}$ of a poset $P$ is defined to be $0$ when $P$ is empty and otherwise,
\[\mathrm{\defin{\rsize{\calL}}} = \frac{|M_1|+\dots+|M_t|}{|P|}.\]
The \defin{relative dimension} of a poset $P$, denoted $\defin{\rdim(P)}$, is the least relative frequency of a local realizer that $P$ admits.
The key difference between local dimension and relative dimension is that the former measures the maximum number of occurrences of elements of $P$ and the latter focuses on the average number of appearances.

Since every realizer is a local realizer, and the relative frequency of a local realizer is always at most the frequency of the local realizer, for every poset $P$, we have
\[\rdim(P) \leq \ldim(P) \leq \dim(P).\]



Relative dimension is less sensitive than the other parameters to a small, complicated subposet. To illustrate this, fix a nonempty poset $Q$ with $\rdim(Q)>2$, and let $Q_n$ be the disjoint union of $Q$ and an $n$-element chain. Then
$\lim_{n\to\infty}\rdim(Q_n)=2$.
Thus the fixed subposet $Q$ contributes a vanishing proportion of the total representation cost as the chain grows. In particular, relative dimension is not monotone under taking subposets.

\section{Results}\label{sec:results}

We first present two families of posets with unbounded local dimension but bounded relative dimension.

A poset $P$ is an \defin{interval order} if there exists a mapping of elements of $P$ to closed proper intervals in $\mathbb{R}$ such that if elements $x$ and $y$ in $P$ are mapped into $[a_x,b_x]$ and $[a_y,b_y]$ respectively, then $x \leq y$ in $P$ if and only if $b_x < a_y$ in $\mathbb{R}$.
In this paper we assume that an interval order is given with its interval representation and we identify elements of the poset with intervals.
For each positive integer $n$, let \defin{$U_n$}, the \defin{canonical interval order}, be the interval order on all the intervals in $\{[a,b]\mid a,b\in[n] \textrm{ with } a<b\}$. 
The dimension of $U_n$ is $\Theta(\log\log n)$ as proved by~F{\"u}redi, Hajnal, R{\"o}dl, and Trotter~\cite{furedi1991interval}, 
while its local dimension is unbounded as proved by Barrera{-}Cruz, Prag, Smith, Taylor, and Trotter~\cite{Barrera.2020}. 
By contrast, we prove the following uniform bound.

\begin{theorem}
  \label{thm:interval}
  For every positive integer $n$, we have $\rdim(U_n) < 4$.
\end{theorem}
We do not know whether the relative dimension of all interval orders is bounded; see \Cref{sec:open-problems}. 

For a nonnegative integer $n$, the \defin{Boolean lattice} of \defin{order} $n$, denoted \defin{$\calB_n$}, is the poset with the ground set consisting of all subsets of $[n]$ ordered by inclusion.
For positive integers $s$ and $t$ with $ s < t \leq n$, we define $\calB_n(s,t)$ as the poset $\calB_n$ restricted to subsets of $[n]$ of cardinality $s$ and $t$.
Following the literature on the other variants of poset dimension, we denote by \defin{$\rdim(s,t;n)$} the value $\rdim(\calB_n(s,t))$.
For a fixed positive integer~$k$, the local dimension of posets in the family $\{\calB_n(1,k) : k \leq n\}$ is unbounded~\cite{Barrera.2020}.
Note that the poset $\calB_n(1,2)$ is isomorphic to the incidence poset of a complete graph on $n$ vertices.\footnote{An \defin{incidence poset} of a graph $G$ is a poset with the set of minimal elements $V(G)$ and the set of maximal elements $E(G)$. The minimal elements form an antichain and the maximal elements form an antichain. We have $v < e$ for $v \in V(G)$ and $e \in E(G)$ whenever $e$ is incident to $v$.}

\begin{theorem}
  \label{thm:incidence_orders}
  For all positive integers $k,n$ with $2\leq k \leq n$, we have
  $\rdim(1,k;n) < k+3$.
\end{theorem}

We next give two settings in which relative dimension is unbounded.

Let $n$ be a positive integer and $p$ be a real number with $0 < p < 1$.
We fix two disjoint $n$-element sets $A_n$ and $B_n$.
A poset $P$ is \defin{$n$-bipartite} if its ground set is $A_n \cup B_n$, the elements in each of $A_n$ and $B_n$ are pairwise incomparable in $P$, and for all $a \in A_n$ and $b \in B_n$, either $a < b$ in $P$ or $a \parallel b$ in $P$.
We denote by $R(n,p)$ a random variable whose values are $n$-bipartite posets drawn from the following probability distribution.
For all $a \in A_n$ and $b \in B_n$, we set $a < b$ in $R(n,p)$ with probability $p$ and the trials for each pair of elements are independent of each other.
Kim et al.~\cite{DBLP:journals/ejc/KimMMSSUW20} proved that $\ldim(R(n,p)) = \Theta(n / \log n)$ a.a.s.\footnote{The abbreviation \q{a.a.s.} stands for \q{asymptotically almost surely}. Given a random variable $X_n$ and an event $E_n$ for each positive integer $n$, we say that $X_n$ \defin{a.a.s.} satisfies $E_n$ if $\lim_{n \rightarrow \infty} P(X_n \in E_n) = 1$.}

\begin{theorem}
  \label{thm:random}
  For a fixed $p$ with $0 < p < 1$, we have $\rdim(R(n,p)) = \Theta(n / \log n )$ asymptotically almost surely.
\end{theorem}

For the last result, we need the following basic property of relative dimension (note that this property also holds for dimension and local dimension~\cite{DBLP:journals/ejc/KimMMSSUW20}).
For two nonempty posets $P$ and $Q$, we consider $P \times Q$ ordered by the \defin{product} order.
Namely, for all $(p,q),(p',q') \in P \times Q$, we have $(p,q) \leq (p',q')$ in $P \times Q$ if and only if $p \leq p'$ in $P$ and $q \leq q'$ in $Q$.

\begin{lemma}
  \label{thm:subadd}
  For any two nonempty posets $P$ and $Q$, $\rdim(P \times Q) \leq \rdim(P) + \rdim(Q)$.
\end{lemma}

To wrap up, we provide a constructive way of obtaining a family of posets with unbounded relative dimension.
It suffices to take the family of Boolean lattices.
We additionally provide a non-trivial upper bound on the relative dimension of Boolean lattices.

\begin{theorem}\label{thm:boolean}
We have $\rdim(\calB_n) \rightarrow \infty$ as $n \rightarrow \infty$. 
Additionally, $\rdim(\calB_n)\leq  \frac{2}{3}n + \frac{1}{3}$ for every positive integer $n$.
\end{theorem}

Note that the exact asymptotic behavior of $\ldim(\calB_n)$ is not known.
The best-known lower bound is $\Omega(n / \log n)$ as proved by Kim et al.~\cite{DBLP:journals/ejc/KimMMSSUW20} and $\ldim(\calB_n) < n$ if $n \geq 4$ as proved by Hodor and Sordyl~\cite{HSlocal}.
On the other hand, $\dim(\calB_n) = n$ for every positive integer $n$.



\section{Proofs}\label{sec:proofs}

\subsection{Interval orders}
In this subsection, we prove~\Cref{thm:interval}.
Before proving the theorem, we state several auxiliary definitions and we prove a technical lemma containing the main idea of the proof.
For an interval order $P$ and real numbers $s$ and $t$ with $s \leq t$, we define the poset $P[s,t]$ as a subposet of $P$ induced by all intervals $[a,b]$ in $P$ with $s \leq a < b \leq t$.
We say that an interval order is \defin{dense} if for all real numbers $s$ and $t$ with $s < t$, the poset $P[s,t]$ contains an antichain of size at least $\frac{1}{2}|P[s,t]|$.
In particular, if an interval order $P$ is dense, then $P[s,t]$ is dense for all real $s$ and $t$ with $s < t$.

\begin{lemma}\label{lemma:dense_interval_orders}
  Let $P$ be a dense interval order.
  Then, there exists a local realizer $\mathcal{M}$ of $P$ of relative frequency less than $4$ such that either $\mathcal{M}$ is empty or $\mathcal{M}$ contains a linear extension of $P$.
  In particular, $\rdim(P) < 4$.
\end{lemma}
\begin{proof}
  We proceed by induction on the number of elements of $P$.
  If $P$ has no elements, then the assertion holds trivially; thus, assume that $P$ has at least one element.
  Let $X$ be an antichain in $P$ of size at least $\frac{1}{2}|P|$.
  We additionally require $X$ to be inclusionwise maximal.
  By the Helly property, there exists $m \in \mathbb{R}$ such that $m \in [a,b]$ for every $[a,b]$ in $X$.
  Let $s$ and $t$ be integers such that $P = P[s,t]$.
  Let $m_-$ be the greatest real number that is the right endpoint of an interval in $P$ with $m_- < m$, and let $m_+$ be the least real number that is the left endpoint of an interval in $P$ with $m < m_+$.
  Note that $m_-$ or $m_+$ may not exist, and in this case, we set them to $s$ or $t$, respectively.
  Let $P_- = P[s,m_-]$ and $P_+ = P[m_+,t]$.
  Note that $X$, the ground set of $P_-$, and the ground set of $P_+$ form a partition of the ground set of $P$.
  Since $P$ has at least one element, $X$ has at least one element, and so, $P_-$ and $P_+$ have both strictly fewer elements than $P$.
  Therefore, by induction, there are $\mathcal{M}_-$ and $\mathcal{M}_+$, local realizers of $P_-$ and $P_+$ respectively, both of relative frequency less than $4$.
  Let $L_-$ and $L_+$ be linear extensions in $\mathcal{M}_-$ and $\mathcal{M}_+$ respectively.
  If either of the realizers is empty, then the corresponding poset is empty, and we just define $L_-$ or $L_+$ as an empty linear extension.

  Let $M_-'$ be a linear order of elements of $X$ where left endpoints sort them, and ties are resolved arbitrarily.
  More precisely, for all $[a,b],[c,d]$ in $X$, if $a < c$, then $[a,b] < [c,d]$ in $M_-'$.
  Next, we insert all the elements of $P_-$ into $M_-'$ as high as possible, obtaining $M_-$. 
  More precisely, for all $[a,b]$ in $P_-$ and $[c,d]$ in $X$, if $[a,b] < [c,d]$ in $M_-$, then $b < c$.
  We symmetrically construct $M_+$. 

  Next, let $L_X$ be $M_-'$ reversed, and let $L$ be the concatenation of $L_-$, $L_X$, and $L_+$.
  In particular, $L$ is a linear extension of $P$.
  Finally, let
  \[\mathcal{M} = \left(\mathcal{M}_- \backslash \{L_-\}\right) \cup \left(\mathcal{M}_+ \backslash \{L_+\}\right) \cup \{M_-,M_+,L\}.\]
  We claim that $\mathcal{M}$ is a local realizer of $P$ of relative frequency less than $4$.
  Let us first verify the latter part of the claim.
  Note that $|P_-| + |P_+| \leq \frac{1}{2}|P| \leq |X|$.
  We have
  \begin{align*}
    \rsize{\calM} \cdot |P| & = \left(\rsize{\calM_-} \cdot |P_-| - |P_-|\right) + \left(\rsize{\calM_+} \cdot |P_+| - |P_+|\right) + |M_-| + |M_+| + |L| \\
                            & < 3|P_-| + 3|P_+| + (|P_-| + |X|) + (|P_+| + |X|) + (|P_-| + |P_+| + |X|)                                                  \\
                            & = 4(|P_-| + |P_+|) + 3|X| + (|P_-| + |P_+|)                                                                                 \\
                            & \leq 4(|P_-| + |P_+| + |X|) = 4|P|.
  \end{align*}
  Therefore, it suffices to show that $\calM$ is indeed a local realizer of $P$.
  Since $L \in \calM$ and $L$ is a linear extension of $P$, it suffices to show that for every pair of incomparable elements in $P$, there are two partial linear extensions in $\calM$ such that in one of them one of the elements is higher and in the other, the other element is higher.
  Let $[a,b]$ and $[c,d]$ be two distinct elements of $P$.
  If both are in $P_-$ or $P_+$, then the above follows from induction (recall that $L_-$ and $L_+$ are now subposets of $L$).
  If both are in $X$, then the order of $[a,b]$ and $[c,d]$ is opposite in $M_-$ and $L$ by definition.
  It is not possible that $[a,b]$ is in $P_-$ and $[c,d]$ is in $P_+$, since then, $b < m < c$, and so, $[a,b] < [c,d]$ in $P$, which is false.
  Thus, assume that $[a,b]$ is in $P_-$ and $[c,d]$ in $X$ (the case of $[a,b]$ in $X$ and $[c,d]$ in $P_+$ is symmetric).
  First, we have $[a,b] < [c,d]$ in $L$.
  On the other hand, since $c \leq b < d$ in $\mathbb{R}$, $[c,d] < [a,b]$ in $M_-$.
  This ends the proof of the lemma.
\end{proof}

\noindent\textbf{Proof of Theorem \ref{thm:interval}.}
Let $n$ be a positive integer.
By~\Cref{lemma:dense_interval_orders}, it suffices to show that $U_n$ is dense.
Let $s$ and $t$ be real numbers with $s\leq t$.
Since all endpoints of intervals in the representation of $U_n$ are positive integers, we may assume that $s$ and $t$ are positive integers.
Let $X$ be the subposet of $U_n[s,t]$ induced by all the intervals containing $m = \left\lfloor(s+t) \slash 2 \right\rfloor$.
Clearly, $X$ is an antichain, hence, it suffices to show that $|X| \geq \frac{1}{2}|U_n[s,t]|$.
The inequality holds if $t = s$; thus, we assume that $s < t$.
Observe that $|X| = (m-s)(t-m) + t-s$.
The inequality
\[|X| = (m-s)(t-m) + t-s \geq \frac{1}{2}\binom{t-s+1}{2} = \frac{1}{2}|U_n[s,t]|\]
can be simplified to
\[-4m^2 + 4m(s+t) \geq (s+t)^2 - 3(t-s).\]
When $m = \frac{s+t}{2}$, we obtain
\[-4m^2 + 4m(s+t) = (s+t)^2,\]
and when $m = \frac{s+t-1}{2}$, we obtain
\[-4m^2 + 4m(s+t) = (s+t)^2 -1.\]
In both cases, the required inequality holds.\hfill$\square$

\subsection{Two levels in a Boolean lattice}
In this subsection, we prove~\Cref{thm:incidence_orders}.

\noindent\textbf{Proof of Theorem \ref{thm:incidence_orders}.} 
Let $k$ and $n$ be positive integers with $2 \leq k \leq n$.
We begin by defining the two linear extensions $L_1$ and $L_2$ of $\calB_n(1,k)$.
Let $L_1$ be obtained in the following way.
We start with the singletons of elements of $[n]$ sorted increasingly.
For every $Y \in \binom{[n]}{k}$, insert $Y$ as low as possible, that is, for every $x \in [n]$, we have $\{x\} < Y$ in $L_1$ when $x \leq \max Y$ and $Y < \{x\}$ in $L_1$ otherwise.
We settle the relations in $L_1$ between elements of $\binom{[n]}{k}$ arbitrarily so that $L_1$ is a total order.
Next, let $L_2$ be obtained symmetrically as follows.
We start with the singletons of elements of $[n]$ sorted decreasingly.
For every $Y \in \binom{[n]}{k}$, insert $Y$ as low as possible, that is, for every $x \in [n]$, we have $\{x\} < Y$ in $L_2$ when $\min Y \leq x$ and $Y < \{x\}$ in $L_2$ otherwise.
We settle the relations in $L_2$ between elements of $\binom{[n]}{k}$ arbitrarily so that $L_2$ is a total order.
Additionally, let $K$ be a partial linear extension of $\calB_n(1,k)$ on all the elements of $\binom{[n]}{k}$ such that the order is reversed with respect to $L_2$.
More precisely, for all $Y,Z \in \binom{[n]}{k}$ with $Y \neq Z$, $Y < Z$ in $K$ if and only if $Z < Y$ in $L_2$.

Let $\ell \in [n-2]$.
We define the family $\calX_\ell$ as all the sets $X \in \binom{[n]}{k}$ such that $\ell \in X$, $\ell+1 \notin X$, and $\ell < \max X$.
Next, we define a partial linear extension $M_\ell$ starting with all singletons $\{x\}$ for integers $x$ with $\ell+1 \leq x \leq n-1$ sorted decreasingly.
We insert all sets from $\calX_\ell$ as low as possible in $M_\ell$, breaking ties arbitrarily.

We claim that $\calM = \{L_1,L_2,K,M_1,\dots,M_{n-2}\}$ is a local realizer of $\calB_n(1,k)$.
Let $x \in [n]$ and $X \in \binom{[n]}{k}$.
Since $\calM$ contains $L_1$, $L_2$, and $K$, it suffices to show that if $x \notin X$, then there exists $M \in \calM$ such that $X < \{x\}$ in $M$.
Thus, we indeed assume that $x \notin X$.
If $x < \min X$, then $X < \{x\}$ in $L_2$.
If $\max X < x$, then $X < \{x\}$ in $L_1$.
Therefore, we assume that $\min X < x < \max X$.
Let $\ell$ and $\ell'$ be consecutive elements in $X$ such that $\ell < x < \ell'$.
Note that in $M_\ell$, we have $X < \{\ell'-1\} \leq \{x\} \leq \{\ell+1\}$.
In particular, $X < \{x\}$ in $M_\ell$, as desired.

We conclude by computing $\rsize{\calM}$.
The singletons $\{1\}$ and $\{n\}$ appear only in $L_1$ and $L_2$.
For an integer $x$ with $2 \leq x \leq n-1$, the singleton $\{x\}$ appears in $L_1$, $L_2$, and $M_{1},\dots,M_{x-1}$, altogether $x+1$ times. 
Every element of $\binom{[n]}{k}$ appears in $L_1$, $L_2$, $K$, and at most $k-1$ remaining partial linear extensions of $\calM$, altogether at most $k+2$ times.
Summarizing, we obtain
\begin{align*}
    \rdim(1,k;n) \leq \rsize{\calM} &\leq \frac{2 + 2+ \sum_{x=2}^{n-1} (x+1) + (k+2)\binom{n}{k} }{n + \binom{n}{k}}\\
    &\leq \frac{1 + \binom{n+1}{2} + (k+2)\binom{n}{k}}{n + \binom{n}{k}}\\
    &\leq \frac{\binom{n+1}{2}}{n + \binom{n}{k}} + (k+2) - \frac{(k+2)n - 1}{n + \binom{n}{k}} < k+3.
\end{align*}
This completes the proof.
\hfill$\square$

\subsection{Random posets}
In this subsection, we study the growth rate of relative dimension of random bipartite posets, namely, we prove~\Cref{thm:random}.
Throughout this subsection, we fix $p$ with $0 < p < 1$.
The upper bound in~\Cref{thm:random} follows from an analogous result on local dimension shown by Kim et al.~\cite{DBLP:journals/ejc/KimMMSSUW20} as local dimension always upper-bounds relative dimension.
Therefore, it suffices to show the lower bound.
First, we need some more notation.

Let $n$ be a positive integer and let $P$ be an $n$-bipartite poset.
We define the \defin{balanced independence number} of $P$ as the maximum nonnegative integer $i$ such that there exist $A' \subset A_n$ and $B' \subset B_n$ with $|A'| = |B'| = i$ such that for all $a \in A'$ and $b \in B'$, we have $a \parallel b$ in $P$.
A pair $(a,b)$ of incomparable elements in $P$ is called \defin{relevant} if $a \in A_n$ and $b \in B_n$.
Within the next two lemmas, we establish two properties of $R(n,p)$ that hold a.a.s.
Next, we will show that these two properties are enough for a poset that can be a value of $R(n,p)$ to have large relative dimension.

Set $c = 1+2\slash (-\log(1-p))$.

\begin{lemma}\label{lem:bin-of-Rnp}
    The balanced independence number of $R(n,p)$ is less than $c \log n$ a.a.s.
\end{lemma}
\begin{proof}
    Let $n$ be a positive integer.
    Let $I_n$ be the random variable that is the balanced independence number of $R(n,p)$.
    Let $m$ be a positive integer.
    For each pair of sets $A' \subset A_n$ and $B' \subset B_n$ with $|A'| = |B'| = m$, we denote by $X_{n,m}(A',B')$ the indicator random variable of the event that for all $a \in A'$ and $b \in B'$, we have $a \parallel b$ in $R(n,p)$.
    Note that $P(X_{n,m}(A',B') = 1) = E(X_{n,m}(A',B')) = (1-p)^{m^2}$.
    Define 
    \[\textstyle X_{n,m} = \sum_{A' \subset A_n, B' \subset B_n, |A'| = |B'| = m} X_{n,m}(A',B').\]
    By linearity of expectation,
    \[E(X_{n,m}) = \binom{n}{m}^2 \cdot (1-p)^{m^2} \leq n^{2m} \cdot (1-p)^{m^2} = 2^{2m \log n + m^2 \log (1-p)}.\]
    Note that $I_n \geq m$ if and only if $X_{n,m} > 0$.
    We show that $X_{n,m} = 0$ for $m \geq c \log n$ a.a.s., which, as noted above, means that the balanced independence number of $R(n,p)$ is less than $c \log n$ a.a.s.
    Since $P(X_{n,m} > 0) \leq E(X_{n,m})$, it suffices to show that if $m \geq c \log n$, then $2m \log n + m^2 \log (1-p) \rightarrow - \infty$ as $n \rightarrow \infty$.
    Note that $c > 0$. 
    Using $\log n \leq m \slash c$, we obtain
    \[2m \log n + m^2 \log (1-p) \leq m^2(2 \slash c + \log(1-p)).\]
    The right-hand side of the above inequality tends to $- \infty$ as long as $2 \slash c + \log(1-p) < 0$.
    Equivalently, $2 \slash (-\log(1-p)) < c$, which is true by definition.
    This completes the proof.
\end{proof}

\begin{lemma}\label{lem:relevant-Rnp}
    The number of relevant incomparable pairs in $R(n,p)$ is at least $(1-p)n^2 \slash 2$ a.a.s.
\end{lemma}
\begin{proof}
    Let $n$ be a positive integer.
    For each pair of elements $a \in A_n$ and $b \in B_n$, we define $X_n(a,b)$ as the indicator function of the event that $a \parallel b$ in $R(n,p)$.
    Note that $X_n = \sum_{a \in A_n, b \in B_n} X_n(a,b)$ is the number of relevant incomparable pairs in $R(n,p)$.
    Observe also that $E(X_n) = (1-p) n^2$.
    Applying Chebyshev's inequality, we obtain $P(X_n > (1-p)n^2 \slash 2) \rightarrow 1$ as $n \rightarrow \infty$, as desired.
\end{proof}

\noindent\textbf{Proof of Theorem \ref{thm:random}.}
In~\cref{lem:bin-of-Rnp,lem:relevant-Rnp}, we proved that the balanced independence number of $R(n,p)$ is less than $c \log n$ a.a.s., and that the number of relevant incomparable pairs in $R(n,p)$ is at least $(1-p)n^2 \slash 2$ a.a.s.
Therefore, it suffices to prove a lower bound on the relative dimension of a poset that can be a value of $R(n,p)$ assuming the above properties.

Let $n$ be a positive integer and let $P$ be an $n$-bipartite poset such that the balanced independence number of $P$ is less than $c \log n$ and the number of relevant incomparable pairs in $P$ is at least $(1-p)n^2 \slash 2$.
Let $\mathcal{M}$ be a local realizer of $P$.
We say that an incomparable pair of elements $(a,b)$ in $P$ is \defin{reversed} in some $M \in \calM$ if $b < a$ in $M$.
For each $M \in \mathcal{M}$, we define \defin{$r(M)$} to be the number of relevant incomparable pairs in $P$ that are reversed in $M$.
Since $\mathcal{M}$ is a local realizer of $P$, every relevant incomparable pair is reversed in some $M \in \calM$.
Therefore, $\sum_{M \in \mathcal{M}}r(M)$ is at least the number of relevant incomparable pairs in $P$, i.e.\ at least $(1-p)n^2 \slash 2$.

On the other hand, we can upper bound $r(M)$ for each $M \in \mathcal{M}$.
Indeed, let $M \in \calM$.
We claim that $r(M) \leq |M| \cdot c\log n$.
Let $\alpha$ be the number of elements of $M$ in $A_n$ and let $\beta$ be the number of elements of $M$ in $B_n$.
Note that $|M| = \alpha + \beta$, and we always have $r(M) \leq \alpha\beta$.
In particular, if $\alpha < c \log n$ or $\beta < c \log n$, then $r(M) < |M| \cdot c \log n$, as desired.
Therefore, we assume that $\alpha\geq c \log n$ and $\beta \geq c \log n$.
Let $a_1,\dots,a_\alpha$ be the elements of $M$ in $A_n$ named so that if $i < j$, then $a_j < a_i$ in $M$.
Let $b_1,\dots,b_\beta$ be the elements of $M$ in $B_n$ named so that if $i < j$, then $b_i < b_j$ in $M$.
Suppose that for $i \in [\alpha]$ and $j \in [\beta]$, the pair $(a_i,b_j)$ is an incomparable pair in $P$ reversed in $M$.
Then, since the balanced independence number of $P$ is less than $c \log n$, we have $i < c \log n$ or $j < c \log n$.
Now, for each $i \in [\alpha]$, let $r(M,i)$ be the number of indices $j \in [\beta]$ such that $(a_i,b_j)$ is an incomparable pair reversed in $M$; note that $r(M,i) \leq \beta$.
In particular, $r(M) = \sum_{i \in [\alpha]} r(M,i)$.
It follows that for each $i \in [\alpha]$ with $i \geq c \log n$, we have $r(M,i) < c \log n$.
Altogether, we obtain
\[r(M) = \sum_{i \in [\alpha]} r(M,i) \leq c \log n \cdot \beta + \alpha \cdot c \log n \leq |M| \cdot c \log n.\]
Summarizing, we have
\[(1-p)n^2 \slash 2 \leq \sum_{M \in \mathcal{M}} r(M) \leq \sum_{M \in \calM} |M| \cdot c \log n.\]
By rearranging the terms, we finally get
\[\frac{(1-p)}{4c} \cdot \frac{n}{\log n} \leq \frac{\sum_{M \in \calM} |M|}{2n} = \rsize{\calM}.\]
This completes the proof.\hfill$\square$

\subsection{Boolean lattice}
In this subsection, we study relative dimension of the Boolean lattices (\Cref{thm:boolean}).
We show that relative dimension can be arbitrarily big but it is strictly less than the order of the lattice.
For the upper bound, we need the subadditivity of relative dimension.
We use it in a similar way as it was already used while studying Boolean dimension of a Boolean lattice~\cite{BHLMM24}.
Namely, it suffices to find a small positive integer $N$ with $\rdim(\calB_{N}) < N$ and exploit the structure of Boolean lattices.

\medskip

\noindent\textbf{Proof of Lemma \ref{thm:subadd}.}
Let $P$ and $Q$ be posets and let $\cgL=\{L_1,\dots, L_s\}$ and $\cgM=\{M_1,\dots, M_t\}$ be local realizers of $P$ and $Q$.
Fix a linear extension $L$ of $P$ and a linear extension $M$ of $Q$.
For each $i \in [s]$, define $L_i'$ as a partial linear extension of $P \times Q$ on the elements $\{(p,q) \mid p \in L_i, q \in Q\}$ with $(p_1,q_1) < (p_2,q_2)$ in $L_i'$ if and only if $p_1 < p_2$ in $L_i$ or $p_1 = p_2$ and $q_1 < q_2$ in $M$.
Dually, define for each $j \in [t]$, a partial linear extension $M_j'$ of $P \times Q$ on the elements $\{(p,q) \mid p \in P, q \in M_j\}$ with $(p_1,q_1) < (p_2,q_2)$ in $M_j'$ if and only if $q_1 < q_2$ in $M_j$ or $q_1 = q_2$ and $p_1 < p_2$ in $L$.
Let $\cgN=\{L_1',\dots, L_s'\} \cup \{M_1',\dots, M_t'\}$.
We claim that $\cgN$ is a local realizer of $P \times Q$.
To see this, let $(p_1,q_1)$ and $(p_2,q_2)$ be two elements of $P \times Q$.
Since $\calL$ is a local realizer of $P$, there is $i \in [s]$ such that $p_1$ and $p_2$ are in the ground set of $L_i$, and therefore, $(p_1,q_1)$ and $(p_2,q_2)$ are in the ground set of $L_i'$.
If $(p_1,q_1) < (p_2,q_2)$ in $P \times Q$, then $(p_1,q_1) < (p_2,q_2)$ in $N$ for every $N \in \calN$ containing these elements by definition.
Now let $(p_1,q_1)$ and $(p_2,q_2)$ be incomparable in $P \times Q$.
It suffices to show that there are $N_1,N_2 \in \calN$ with $(p_1,q_1) > (p_2,q_2)$ in $N_1$ and $(p_1,q_1) < (p_2,q_2)$ in $N_2$.
We show the existence of $N_1$.
The existence of $N_2$ follows by a symmetric argument.
If $p_1$ and $p_2$ are incomparable in $P$, there is $i \in [s]$ with $p_1 > p_2$ in $L_i$ and so $(p_1,q_1) > (p_2,q_2)$ in $L_i'$.
Dually, if $q_1$ and $q_2$ are incomparable in $Q$, there is $j \in [t]$ with $q_1 > q_2$ in $M_j$ and so $(p_1,q_1) > (p_2,q_2)$ in $M_j'$.
If $p_1$ and $p_2$ are comparable in $P$ and $q_1$ and $q_2$ are comparable in $Q$, it follows that $p_1 < p_2$ and $q_1 > q_2$ or $p_1 > p_2$ and $q_1 < q_2$.
In the former case, there is $j \in [t]$ with $q_1 > q_2$ in $M_j$ and thus $(p_1,q_1) > (p_2,q_2)$ in $M_j'$.
In the latter case, there is $i \in [s]$ with $p_1 > p_2$ in $L_i$ and thus $(p_1,q_1) > (p_2,q_2)$ in $L_i'$.
Therefore, $\cgN$ is indeed a local realizer of $P \times Q$.
Finally, suppose that $\rdim(P) = \|\calL\|_r$ and $\rdim(Q) = \|\calM\|_r$ and observe
\begin{align*}
  \rdim(P \times Q) \leq \|\calN\|_r & = \sum_{N \in \cgN} \frac{|N|}{|P|\cdot |Q|} = \sum_{i \in [s]} \frac{|L_i|\cdot |Q|}{|P|\cdot |Q|}+\sum_{j \in [t]}\frac{|M_j|\cdot |P|}{|P|\cdot |Q|} \\
                                     & = \sum_{i \in [s]} \frac{|L_i|}{|P|}+\sum_{j \in [t]}\frac{|M_j|}{|Q|}  = \rdim(P) + \rdim(Q).
\end{align*}
This ends the proof. \hfill$\square$

\medskip

To prove the upper bound in \cref{thm:boolean}, we use the following result on the existence of a Boolean lattice as a subset of a set system.
\begin{theorem}[Gunderson, Rödl, and Sidorenko~\cite{Gunderson.1999}]
  \label{thm:boolean_helper}
  Let $n$ be a positive integer and $X$ be an $n$-element set.
  For all positive integers $d$ and $k$ with
  \[k \geq 10^d2^{-1/2^{d-1}}d^{d-1/2^d} \cdot n^{-1/2^d}\cdot 2^n,\]
 every family $\calF$ of subsets of $X$ with $|\calF| > k$ contains a subfamily isomorphic to $\calB_d$ as an inclusion-ordered poset.
\end{theorem}

\begin{figure}[t]
  \hfill
  \includegraphics[valign=c]{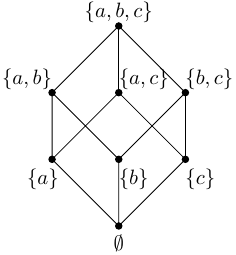}
  \hfill
  \begin{minipage}[c]{.65\linewidth}
    \begin{itemize}
      \item[$L_1:$] $\emptyset < \{a\} < \{b\} < \{a,b\} < \{c\} < \{a,c\} \\\phantom{\emptyset}< \{b,c\} < \{a,b,c\}$
      \item[$L_2:$] $\{c\} < \{b\} < \{b,c\} < \{a\} < \{a,c\} < \{a,b\}$
      \item[$L_3:$] $\{a,c\} < \{b\}$
    \end{itemize}
  \end{minipage}
  \hfill\null

  \caption{$\calB_3$ and a local realizer that witnesses its relative dimension 2.}
  \label{fig:B3}
\end{figure}

\smallskip
\noindent
\textbf{Proof of Theorem~\ref{thm:boolean}.}
We first prove the lower bound on the relative dimension of Boolean lattices.
Let $n$ be a large integer (specified later) and let $\cgM=\{M_1,\dots ,M_t\}$ be a local realizer of $\calB_n$.
For each element $A \subset [n]$, let $c(A)$ be the number of partial linear extensions in $\cgM$ that contain $A$.
Let $k=2\cdot \|\calM\|_r$ and let $B'$ be the poset that results from removing all the elements $A \subset [n]$ with $c(A) \geq k$ from $\calB_n$.
Then,
\begin{align*}
  \|\calM\|_r & = \sum_{i \in [t]}\frac{|M_i|}{|\calB_n|}     = \sum_{A \subset [n]}\frac{c(A)}{|\calB_n|} \geq \sum_{\substack{A \subset [n] \\c(A)\geq k}}\frac{k}{|\calB_n|}\\
              & =(|\calB_n|-|B'|)\cdot \frac{k}{|\calB_n|} = (2^n-|B'|)\cdot \frac{\|\calM\|_r}{2^{n-1}}.
\end{align*}
Therefore, we obtain $|B'| \geq 2^{n-1}$.

Fix $d= \left\lfloor\log\log((n/2)^{1/2})-\log\log\log((n/2)^{1/2})-\log\log\log\log((n/2)^{1/2})\right\rfloor$.
By Theorem \ref{thm:boolean_helper}, $B'$ contains a subposet isomorphic to $\calB_d$ as for $n$ large enough, we have
$$|B'|\geq 2^{n-1} \geq 10^d2^{-1/2^{d-1}}d^{d-1/2^d} \cdot n^{-1/2^d}\cdot 2^n.$$
The partial linear extensions in $\cgM$ restricted to the elements of $B'$ form a local realizer $\calM'$ with $\|\calM'\| \leq 2 \cdot \|\calM\|_r$.
Since $B'$ contains $\calB_d$, by~\cite[Theorem~6]{DBLP:journals/ejc/KimMMSSUW20} we have
\[\|\calM'\| \geq \ldim(B') \geq \ldim(\calB_d) \geq  \frac{d}{2e \log (d)}.\]
Altogether, assuming that $\calM$ is optimal, we obtain
\[\rdim(\calB_n) = \|\calM\|_r \geq \frac{\|\calM'\|}{2} \geq \frac{d}{4e \log (d)} \geq c \cdot  \frac{\log\log n}{\log\log\log n}\]
where $c$ is an absolute constant (independent of $n$).

Finally, we proceed with the upper bound.
First, observe that $\rdim(\calB_1) \leq 1$ and $\rdim(\calB_2) \leq 3 \slash 2$.
Next, note that $\rdim(\calB_3) \leq 2$ as witnessed by a local realizer in~\cref{fig:B3}.
Let $n$ be an integer with $n \geq 3$.
Since $\calB_n = \calB_3 \times \dots \times \calB_3 \times \calB_r = (\calB_3)^{\lfloor n \slash 3\rfloor} \times \calB_r$ where $n \equiv r \bmod 3$ and $r \in \{0,1,2\}$, by~\Cref{thm:subadd}, we have
$$\rdim(\calB_n) \leq \begin{cases}
    n/3 \cdot \rdim(\calB_3)                              & \text{if } n \equiv 0 \mod 3, \\
    \lfloor n/3\rfloor\cdot \rdim(\calB_3)+\rdim(\calB_1) & \text{if }n \equiv 1 \mod 3,  \\
    \lfloor n/3\rfloor\cdot \rdim(\calB_3)+\rdim(\calB_2) & \text{if }n \equiv 2 \mod 3,
  \end{cases}$$
which gives the bound. \hfill$\square$

\section{Open Problems}
\label{sec:open-problems}

As the relative dimension is a new concept, there is no prior research on the behavior of the parameter.
Thus, we want to point out a few directions, which are, in our opinion, worthwhile to investigate further.
For canonical interval orders, we have shown that the relative dimension is bounded. If relative dimension were monotone under taking subposets, this would imply that the relative dimension of all interval orders is bounded. However, this is not the case, which makes it an interesting question to explore further.

\smallskip
\noindent
\textbf{Question 1:} Is the relative dimension bounded for the family of \textbf{all} interval orders?

\smallskip
\noindent
We showed that the relative dimension of the Boolean lattice is unbounded.
However, the lower bound that we provide is very far from our upper bound.
It is thus desirable to achieve better bounds and narrow this gap, which gives rise to the next question.

\smallskip
\noindent
\textbf{Question 2:} What is the asymptotic behavior of relative dimension of Boolean lattices?

\smallskip
\noindent
For posets with a planar diagram, it is well-known that both dimension and local dimension are unbounded~\cite{Kelly.1981,BGT20}, i.e., for any integer $k$, there is a poset with a planar diagram such that its (local) dimension is larger than $k$.
It is thus natural to ask whether the relative dimension of posets with planar diagrams is unbounded as well.
A good starting point could be to investigate the construction showing that local dimension of posets with planar diagrams is unbounded~\cite{BGT20}.

\smallskip
\noindent
\textbf{Question 3:} Is the relative dimension of posets with planar diagrams bounded? 

\smallskip
\noindent
Let $P$ be a poset with a connected cover graph.
We call a \defin{block} of $P$ a subposet induced by a maximal 2-connected subset of the cover graph.
If every block of $P$ has dimension at most $d$, then the dimension of $P$ is at most $d+2$~\cite{Trotter.2015}.
The same is not true for local dimension: the local dimension of a poset is not bounded with respect to the local dimension of its blocks.
It is thus interesting to investigate whether a similar result holds for the relative dimension.

\smallskip
\noindent
\textbf{Question 4:} Is there a function $f$ such that for every poset $P$, if all its blocks have relative dimension at most $d$, then the relative dimension of $P$ is at most $f(d)$?

\smallskip
\noindent
We noted that the relative dimension of a poset is bounded from above by the local dimension.
Furthermore, we showed that the two can be arbitrarily far away from each other, as there are posets like the canonical interval orders for which the relative dimension is bounded but the local dimension is unbounded.
There is a family of posets with bounded local dimension, and thus bounded relative dimension, but unbounded Boolean dimension~\cite{Trotter.2017}.
For the other direction, there is also a family of posets with bounded Boolean dimension but unbounded local dimension.
The only remaining question requiring an answer in this line of research is the following.

\smallskip
\noindent
\textbf{Question 5:} Is there a family of posets $P_n$ with $\bdim(P_n) < c$ and $\rdim(P_n) \to \infty$?

\bibliographystyle{amsplain}
\bibliography{paper.bib}

\end{document}

%% file: macros.tex
\usepackage{graphicx}
\usepackage[export]{adjustbox}
\usepackage{cleveref}

\colorlet{defcolor}{ForestGreen}
\newcommand{\defin}[1]{\relax\ifmmode{\color{defcolor}{#1}}\else{\emph{\textcolor{defcolor}{#1}}}\fi}

\newcommand{\q}[1]{``#1''}

\newtheorem{theorem}{Theorem}
\newtheorem{lemma}[theorem]{Lemma}

\theoremstyle{definition}

\theoremstyle{remark}

\newcommand{\N}{\mathbb{N}}

\newcommand{\cgA}{\mathcal{A}}

\newcommand{\cgB}{\mathcal{B}}
\newcommand{\cgC}{\mathcal{C}}

\newcommand{\cgL}{\mathcal{L}}
\newcommand{\cgM}{\mathcal{M}}
\newcommand{\cgN}{\mathcal{N}}

\newcommand{\Oh}{\mathcal{O}}

\newcommand{\calB}{\mathcal{B}}

\newcommand{\calF}{\mathcal{F}}

\newcommand{\calL}{\mathcal{L}}
\newcommand{\calM}{\mathcal{M}}
\newcommand{\calN}{\mathcal{N}}

\newcommand{\calX}{\mathcal{X}}

\newcommand{\bfthree}{\mathbf{3}}

\newcommand{\ldim}{\operatorname{ldim}}
\newcommand{\rdim}{\operatorname{rdim}}
\newcommand{\fdim}{\operatorname{fdim}}
\newcommand{\bdim}{\operatorname{bdim}}
\newcommand{\symdiff}{\operatorname {\Delta }}

\let\subset\subseteq
\let\cref\Cref

\let\le\leqslant

\let\leq\leqslant
\let\geq\geqslant

%% file: paper.bib
@article{DBLP:journals/ejc/KimMMSSUW20,
  author    = {Jinha Kim and
               Ryan R. Martin and
               Tom{\'{a}}{\v s} Masa{\v r}{\'{\i}}k and
               Warren Shull and
               Heather C. Smith and
               Andrew J. Uzzell and
               Zhiyu Wang},
  title     = {On difference graphs and the local dimension of posets},
  journal   = {European Journal of Combinatorics},
  volume    = {86},
  pages     = {103074},
  year      = {2020},
  url       = {https://doi.org/10.1016/j.ejc.2019.103074},
  doi       = {10.1016/J.EJC.2019.103074},
  bibsource = {dblp computer science bibliography, https://dblp.org},
        note = {\href{http://arxiv.org/abs/1803.08641}{arXiv:1803.08641}}
}

@article{DusMil41,
  title     = {Partially ordered sets},
  author    = {Dushnik, Ben and Miller, Edwin W.},
  journal   = {American Journal of Mathematics},
  volume    = {63},
  number    = {3},
  pages     = {600--610},
  year      = {1941},
  publisher = {JSTOR}
}

@inproceedings{furedi1991interval,
  title     = {Interval orders and shift graphs},
  author    = {F{\"u}redi, Zoltán and Hajnal, P{\'e}ter and R{\"o}dl, Vojt\v{e}ch and Trotter, William T.},
  booktitle = {Colloquium Mathematical
Society J{\'a}nos Bolyai},
  volume    = {60},
  pages     = {297--313},
  year      = {1991}
}

@article{Barrera.2020,
  author    = {Fidel Barrera{-}Cruz and
               Thomas Prag and
               Heather C. Smith and
               Libby Taylor and
               William T. Trotter},
  title     = {Comparing {D}ushnik-{M}iller Dimension, {B}oolean Dimension and Local Dimension},
  journal   = {Order},
  volume    = {37},
  number    = {2},
  pages     = {243--269},
  year      = {2020},
  url       = {https://doi.org/10.1007/s11083-019-09502-6},
  doi       = {10.1007/S11083-019-09502-6},
  bibsource = {dblp computer science bibliography, https://dblp.org},
        note = {\href{http://arxiv.org/abs/1710.09467}{arXiv:1710.09467}}
}

@article{Erdös.1991,
  author    = {Paul Erd{\H{o}}s and
               Henry A. Kierstead and
               William T. Trotter},
  title     = {The Dimension of Random Ordered Sets},
  journal   = {Random Structures and Algorithms},
  volume    = {2},
  number    = {3},
  pages     = {253--275},
  year      = {1991},
  url       = {https://doi.org/10.1002/rsa.3240020302},
  doi       = {10.1002/RSA.3240020302},
  bibsource = {dblp computer science bibliography, https://dblp.org}
}

@article{Kelly.1981,
  author    = {David Kelly},
  title     = {On the dimension of partially ordered sets},
  journal   = {Discrete Mathematics},
  volume    = {35},
  number    = {1--3},
  pages     = {135--156},
  year      = {1981},
  url       = {https://doi.org/10.1016/0012-365X(81)90203-X},
  doi       = {10.1016/0012-365X(81)90203-X},
  bibsource = {dblp computer science bibliography, https://dblp.org}
}

@inbook{Trotter.2015, 
place={Cambridge}, 
title={Dimension and Cut Vertices: An Application of {R}amsey Theory}, 
booktitle={Connections in Discrete Mathematics: A Celebration of the Work of Ron Graham}, 
publisher={Cambridge University Press}, 
author={Trotter, William T. and Walczak, Bartosz and Wang, Ruidong}, 
year={2018}, 
pages={187–-199},
note = {\href{http://arxiv.org/abs/1505.08162}{arXiv:1505.08162}}
}

@article{Trotter.2017,
  author    = {William T. Trotter and
               Bartosz Walczak},
  title     = {Boolean Dimension and Local Dimension},
  journal   = {Electronic Notes in Discrete Mathematics},
  volume    = {61},
  pages     = {1047--1053},
  year      = {2017},
  url       = {https://doi.org/10.1016/j.endm.2017.07.071},
  doi       = {10.1016/J.ENDM.2017.07.071},
  bibsource = {dblp computer science bibliography, https://dblp.org},
        note = {\href{http://arxiv.org/abs/1705.09167}{arXiv:1705.09167}}
}

@article{Gunderson.1999,
  author    = {David S. Gunderson and
               Vojtech R{\"{o}}dl and
               Alexander F. Sidorenko},
  title     = {Extremal Problems for Sets Forming {B}oolean Algebras and Complete Partite
               Hypergraphs},
  journal   = {Journal of Combinatorial Theory {A}},
  volume    = {88},
  number    = {2},
  pages     = {342--367},
  year      = {1999},
  url       = {https://doi.org/10.1006/jcta.1999.2973},
  doi       = {10.1006/JCTA.1999.2973},
  bibsource = {dblp computer science bibliography, https://dblp.org}
}

@article{BGT20,
	doi = {10.37236/9258},
	url = {https://doi.org/10.37236/9258},
	year = {2020},
	publisher = {The Electronic Journal of Combinatorics},
	volume = {27},
	number = {4},
	author = {Bart{\l}omiej Bosek and Jaros{\l}aw Grytczuk and William T. Trotter},
	title = {Local Dimension is Unbounded for Planar Posets},
	journal = {The Electronic Journal of Combinatorics},
        note={\href{http://arxiv.org/abs/1712.06099}{arXiv:1712.06099}}
}

@article{BHLMM24,
  title     = "Boolean dimension of a {B}oolean lattice",
  author    = "Bria{\'n}ski, Marcin and Hodor, J\k{e}drzej and La, Hoang and Micek,
               Piotr and Michno, Katzper",
  journal   = "Order",
  publisher = "Springer Science and Business Media LLC",
  volume    =  42,
  number    =  1,
  pages     = "25--36",
  month     =  apr,
  year      =  2025,
  copyright = "https://www.springernature.com/gp/researchers/text-and-data-mining",
  note = {\href{http://arxiv.org/abs/2307.16671}{arXiv:2307.16671}}
}

@Inbook{NP89,
author="Ne{\v{s}}et{\v{r}}il, Jaroslav
and Pudl{\'a}k, Pavel",
title="A Note on {B}oolean Dimension of Posets",
bookTitle="Irregularities of Partitions",
year="1989",
publisher="Springer Berlin Heidelberg",
address="Berlin, Heidelberg",
pages="137--140",
isbn="978-3-642-61324-1",
doi="10.1007/978-3-642-61324-1_12",
url="https://doi.org/10.1007/978-3-642-61324-1_12"
}

@misc{ Ueckerdt16,
  title={Local dimension}, 
    note={concept introduced and developed in Order \& Geometry Workshop, {G}u{\l}towy},
  author={Torsten Ueckerdt},
  year={2016},
}

@misc{HSlocal,
      title={Local dimension of a {B}oolean lattice}, 
      author={J\k{e}drzej Hodor and Jakub Sordyl},
      year={2025},
      eprint={2512.10413},
      archivePrefix={arXiv},
      primaryClass={math.CO},
      url={https://arxiv.org/abs/2512.10413}, 
  note = {\href{http://arxiv.org/abs/2512.10413}{arXiv:2512.10413}}
}
